\documentclass[11pt,a4paper]{article}
\usepackage{amsmath,bm,graphicx,booktabs}
\usepackage{geometry,float,lmodern}
\usepackage{multirow}
\usepackage[authoryear]{natbib}
\usepackage{algorithm,algorithmic}
\usepackage{caption}
\usepackage[colorlinks,citecolor=blue,urlcolor=blue]{hyperref}
\usepackage{setspace}
\usepackage{inputenc}
\usepackage[english]{babel}
\usepackage{amsthm}
\usepackage{color}
\usepackage{courier}
\usepackage{amssymb}
\usepackage{mathrsfs}
\usepackage{enumitem}
\usepackage[title,toc,page]{appendix}

\allowbreak

\providecommand{\keywords}[1]{\textbf{\textit{key words: }} #1}

\newtheorem{theorem}{Theorem}
\newtheorem{proposition}{Proposition}
\newtheorem{remark}{Remark}

\newtheorem{lemma}{Lemma}

\newcommand{\pcite}[1]{\citeauthor{#1}'s (\citeyear{#1})}

\title{Geometric ergodicity of P{\'o}lya-Gamma Gibbs sampler for Bayesian logistic regression with a flat prior }

\date{}

\author{ Xin Wang\thanks{Email: xinwang@iasate.edu}  \, and Vivekananda Roy\thanks{Email: vroy@iastate.edu} \\
Department of Statistics, Iowa State University, Ames, IA}

\begin{document}

\maketitle

\begin{abstract}
The Logistic regression model is the most popular model for analyzing
  binary data. In the absence of any prior information, an improper
  flat prior is often used for the regression coefficients in Bayesian
  logistic regression models. The resulting intractable posterior
  density can be explored by running \pcite{polson2013bayesian} data
  augmentation (DA) algorithm. In this paper, we establish that the
  Markov chain underlying \pcite{polson2013bayesian} DA algorithm is
  geometrically ergodic. Proving this theoretical result is
  practically important as it ensures the existence of central limit
  theorems (CLTs) for sample averages under a finite second moment
  condition. The CLT in turn allows users of the DA algorithm to
  calculate standard errors for posterior estimates.
\end{abstract}


\section{Introduction}
\label{sec:intro_logistic}

Let $(Y_1, Y_2, \dots, Y_n)$ denote the vector of Bernoulli random
variables and $\bm{x}_i$ be the $p \times 1$ vector of known covariates
associated with the $i$th observation for $i=1,\dots,n$.  Let
$\bm{\beta} \in \mathbb{R}^p$ be the unknown vector of regression
coefficients. A generalized linear model can be built
\citep{mcculloch2011model} with a link function that connects the
expectation of $Y_i$ with the covariate $\bm{x}_i$. One
popular link function is the logit link function, $F^{-1}(\cdot)$,
where $F$ is the cumulative distribution function of the standard logistic random variable, that is
$F(t) \equiv e^t/(1+e^t)$ for $t\in \mathbb{R}$. The logit link
function leads to the logistic regression model,
\[
F^{-1}\left(P(Y_i =1)\right) = \log \Bigg(\frac{P(Y_i =1)}{1-P(Y_i =1)}\Bigg)=\bm{x}_i^T\bm{\beta}.
\]
The popularity of the logistic regression model is due to the fact that
$P(Y_i = 1)$ has a closed form expression of
$\bm{x}_i^T\bm{\beta}$, and it is easy to interpret $\bm{\beta}$
in terms of odds ratio.

Let $\bm{y} = (y_1, y_2, \dots, y_n)^T$ be the vector of observed
Bernoulli response variables. The likelihood function for $\bm{\beta}$
is
\[
L\left(\bm{\beta}|\bm{y}\right)=\prod_{i=1}^{n}\frac{\left[\exp\left(\bm{x}_{i}^{T}\bm{\beta}\right)\right]^{y_{i}}}{1+\exp\left(\bm{x}_{i}^{T}\bm{\beta}\right)}.
\]

In a Bayesian framework, when there is no prior information available
about the parameters, noninformative priors are generally used. A
popular method of analyzing binary data is by fitting a Bayesian
logistic regression model with a flat prior on $\bm{\beta}$. If the prior density of $\bm{\beta}$, 
$\pi\left(\bm{\beta}\right)\propto1$, the posterior density of
$\bm{\beta}$ is
\begin{equation}
\label{eq_marginal}
\pi\left(\bm{\beta}|\bm{y}\right)=\frac{L\left(\bm{\beta}|\bm{y}\right) \pi\left(\bm{\beta}\right)}{c\left(\bm{y}\right)}=\frac{1}{c\left(\bm{y}\right)}\prod_{i=1}^{n}\frac{\left[\exp\left(\bm{x}_{i}^{T}\bm{\beta}\right)\right]^{y_{i}}}{1+\exp\left(\bm{x}_{i}^{T}\bm{\beta}\right)},
\end{equation}
provided the marginal density
\[
c\left(\bm{y}\right)=\int_{\mathbb{R}^{p}}\prod_{i=1}^{n}\frac{\left[\exp\left(\bm{x}_{i}^{T}\bm{\beta}\right)\right]^{y_{i}}}{1+\exp\left(\bm{x}_{i}^{T}\bm{\beta}\right)}d\bm{\beta} < \infty.
\]
\cite{chen2001propriety} discuss the necessary and sufficient conditions
for propriety of the posterior density \eqref{eq_marginal}, that is, when
$c(\bm{y}) < \infty$. These conditions are given in 
  \ref{app_chen_condition}.  Throughout this paper, we assume that
the posterior density \eqref{eq_marginal} is proper.

From \eqref{eq_marginal}, we know that the posterior density of
$\bm{\beta}$, $\bm{\pi}(\bm{\beta}\vert \bm{y})$, is intractable in
the sense that means with respect to this density are not available in
closed form. Markov chain Monte Carlo (MCMC) algorithms are generally
used for exploring this posterior density. The data augmentation (DA)
algorithm proposed in \cite{albert1993bayesian} for the Bayesian probit
regression model is widely used. For the logistic regression model,
there have been many attempts to produce such a
DA algorithm \citep{holmes2006bayesian, fruhwirth2010data}. Recently, \cite{polson2013bayesian} (denoted as PS\&W hereafter) have proposed a
new DA algorithm for the logistic regression model based on latent variables following the  P{\'o}lya-Gamma (PG) distribution. As mentioned in \cite{choi2013polya}, PS\&W's algorithm is the first DA algorithm for the logistic regression that is truly
analogous to \pcite{albert1993bayesian} DA algorithm. PS\&W's DA algorithm, like \pcite{albert1993bayesian} DA
for the probit model, in every iteration makes two draws --- one draw from
a $p-$dimensional normal distribution for $\bm{\beta}$ and the other
draw for the latent variables. We now describe these two steps.

Let $\bm{X}$ denote the $n\times p$ design matrix with $i$th row
$\bm{x}_i^T$. Let $\mathbb{R}_{+} = (0,\infty)$ and for
$\left( \omega_1,\omega_2,\dots,\omega_n\right) \in \mathbb{R}_+^n$,
define $\bm{\Omega}$ to be the $n\times n$ diagonal matrix with $i$th
diagonal element $\omega_i$. Finally let $PG(1,b)$ denote the
P\'olya-Gamma distribution defined in Section \ref{sec_poly} with
parameters 1 and $b$. A single iteration of PS\&W's algorithm uses the following
two steps to move from $\bm{\beta}^{\prime}$ to $\bm{\beta}$.

\begin{algorithm}[H]
	\caption*{{\bf{PS\&W's algorithm}:}}
	\begin{algorithmic}[1]
			\STATE Draw $\omega_1,\dots, \omega_n$ independently with $\omega_{i}\sim\text{PG}\left(1,\left|\bm{x}_i^{T}\bm{\beta}^{\prime}\right|\right)$.
		\STATE Draw $\bm{\beta}\sim N\left(\left(\bm{X}^{T}\bm{\Omega}\bm{X}\right)^{-1}\bm{X}^{T}\bm{\kappa},\left(\bm{X}^{T}\bm{\Omega}\bm{X}\right)^{-1}\right)$, where $\bm{\kappa} = \left(\kappa_1,\dots,\kappa_n\right)^T$ with $\kappa_i  = y_i - 1/2$.
	\end{algorithmic}
\end{algorithm}

PS\&W provided an efficient method for sampling from the P{\'o}lya-Gamma
distribution. It can be shown that the transition density of the
 Markov chain underlying the above DA algorithm is strictly positive
everywhere, which implies the chain is Harris ergodic
\citep{asmussen2011new}. Thus the sample averages based on the DA
chain can be used to consistently estimate posterior means. However,
in order to provide standard errors for these estimates one needs to
show the existence of Markov chain CLTs for these
estimators. A standard method of establishing Markov chain CLT
is by proving the chain to be {\it geometrically ergodic}
\citep{roberts1997geometric}. Geometric ergodicity also allows consistent estimation of asymptotic variances in Markov chain CLTs by batch means or spectral variance methods
\citep{flegal2010batch}. This in turn allows the MCMC users to decide
how long to run MCMC simulations \citep{jones2001honest}. Thus proving
geometric ergodicity has important practical benefits. In this paper,
we prove that the Markov chain underlying PS\&W's DA algorithm is geometrically ergodic.

\cite{choi2013polya} considered normal priors on the regression
parameters and proved uniform ergodicity of the corresponding
P{\'o}lya-Gamma DA Markov chain by establishing a minorization
condition.  \cite{choi2017analysis} considered the one-way logistic
ANOVA model under a flat prior on group (treatment) main effects and
showed that the Markov operator corresponding to
P{\'o}lya-Gamma sampler is trace-class. The assumption of the one-way logistic
ANOVA model is restrictive and has limited applications. Here, we analyze the convergence rate
of PS\&W's DA algorithm for Bayesian logistic regression models with a general form of the design
matrix under a flat prior on regression coefficients. In particular, we establish that PS\&W's DA algorithm for the Bayesian logistic regression model under the improper flat prior is always geometrically ergodic. The conditions we need are only the conditions of Proposition \ref{prop1}
in \ref{app_chen_condition}, which guarantee the posterior
propriety. Since we use {\it drift condition} to prove geometric
ergodicity of the DA algorithm and hence CLTs for sample average estimators, the techniques used here are
different from that of \cite{choi2013polya} and
\cite{choi2017analysis}.

The rest of the paper is organized as follows. In
Section~\ref{sec_poly}, we describe PS\&W's Gibbs
sampler. Section~\ref{sec:ge} contains a brief discussion on geometric
rate of convergence for Markov chains and a proof of geometric
ergodicity of PS\&W's Gibbs sampler. Some concluding remarks are
given in Section~\ref{sec:disc}. Finally, the appendix contains some
technical results.

\section{PS\&W's Gibbs sampler}
\label{sec_poly}

In PS\&W's DA algorithm, latent variables with the P\'olya-Gamma
distribution are introduced. The probability density function for a
P\'olya-Gamma random variable with parameters $a >0$ and $b \geq 0 $
is,
\begin{equation}
\label{eq_polya}
 f\left(w|a,b\right)=\cosh^a\left(b/2\right)\frac{2^{a-1}}{\Gamma(a)}\sum_{n=0}^{\infty}\left(-1\right)^{n}\frac{\Gamma(n+a)}{\Gamma(n+1)}\frac{\left(2n+a\right)}{\sqrt{2\text{\ensuremath{\pi}}w^{3}}}e^{-\frac{\left(2n+a\right)^{2}}{8w}-\frac{b^{2}}{2}w}, w>0.
 \end{equation}
 We write $W\sim PG(a,b)$. (Recall that the hyperbolic cosine function $\cosh$ is defined as $\cosh(t) = \left( e^t + e^{-t}\right)/2$.)

 \cite{choi2013polya} developed a new way to formulate PS\&W's DA
 algorithm, which we briefly describe now. Let
 $\bm{\omega} = (\omega_1,\dots,\omega_n)^T$ be the latent
 variables. Assume that, conditional on $\bm{\beta}$, $Y_i$ and
 $\omega_i$ are independent with
 $Y_i \sim \text{Bernoulli} (F(\bm{x}_i^T\bm{\beta}))$ and
 $\omega_i \sim PG(1,\vert \bm{x}_i^T\bm{\beta} \vert)$. Also,
 conditional on $\bm{\beta}$, let $\{(Y_i, \omega_i), i=1,\dots,n\}$
 be $n$ independent pairs. Then the {\it complete posterior} density
 of $\bm{\beta}$ and $\bm{\omega}$ is
\begin{equation}
\label{eq_joint1}
\pi\left(\bm{\beta},\bm{\omega}|\bm{y}\right) =  \frac{\left[\prod_{i=1}^n P(Y_i = y_i\vert \bm{\beta})\right] \left[\prod_{i=1}^n f(\omega_i\vert 1, \vert \bm{x}_i^T\bm{\beta} \vert)\right] \pi(\bm{\beta})}{c(\bm{y})}.
\end{equation}
Clearly from \eqref{eq_marginal} we see that,
\[
\int_{\mathbb{R}^{n}}\pi\left(\bm{\beta},\bm{\omega}|\bm{y}\right)d\bm{\omega}=\pi\left(\bm{\beta}|\bm{y}\right),
\]
that is, the $\bm{\beta}$ marginal density of the augmented
posterior density $\pi(\bm{\beta},\bm{\omega}\vert \bm{y})$ is our
target posterior density $\pi(\bm{\beta}\vert \bm{y})$. 

Let
 $p\left(\omega_{i}\right)$ be the probability function of $PG(1,0)$
 and $\kappa_{i}=y_{i}-1/2$, as defined before. It can be checked that,
\begin{equation}
\label{eq_joint}
\pi\left(\bm{\beta},\bm{\omega}|\bm{y}\right) \propto   \prod_{i=1}^{n}\exp\left[\kappa_{i}\bm{x}_{i}^{T}\bm{\beta}-\omega_{i}\left(\bm{x}_{i}^{T}\bm{\beta}\right)^{2}/2\right]p\left(\omega_{i}\right).
\end{equation}
 PS\&W's DA algorithm is simply a two-variable
Gibbs sampler that, in each iteration, alternates draws from the two
 conditional distributions of
$\pi(\bm{\beta},\bm{\omega}\vert \bm{y})$. Below we present the
conditional densities of $\bm{\omega}$ given $\bm{\beta}$, $\bm{y}$ and $\bm{\beta}$ given $\bm{\omega}$, $\bm{y}$.

From \eqref{eq_joint1} we see that
\begin{equation}
\label{eq:pg_omega}
\omega_{i}|\bm{\beta}, \bm{y}  \overset{\text{ind}}\sim\text{PG}\left(1,\left|\bm{x}_i^{T}\bm{\beta}\right|\right), \text{for } i=1,\dots,n, 
\end{equation}
that is, the conditional distribution of $\bm{\omega}$ given $\bm{\beta}$, $\bm{y}$ is independent of $\bm{y}$.
Thus the conditional density of $\bm{\omega}$ given $\bm{\beta}$, $\bm{y}$ is 
\begin{equation}
\label{eq:cond_omega}
\pi\left(\bm{\omega}|\bm{\beta},\bm{y}\right)\propto\prod_{i=1}^{n}\exp\left[-\omega_{i}\left(\bm{x}_{i}^{T}\bm{\beta}\right)^{2}/2\right]p\left(\omega_{i}\right).
\end{equation}
 From
\eqref{eq_joint}, it is easy to see that the conditional density
of $\bm{\beta}$ is
\begin{equation}
\label{eq:cond_beta}
\pi\left(\bm{\beta}|\bm{\omega},\bm{y}\right)\propto\exp\left[-\frac{1}{2}\bm{\beta}^{T}\bm{X}^{T}\bm{\Omega}\bm{X}\bm{\beta}+\bm{\beta}^{T}\bm{X}^{T}\bm{\kappa}\right],
\end{equation}
where $\bm{\kappa}=\left(\kappa_{1},\dots,\kappa_{n}\right)^{T}$. Thus the conditional distribution of $\bm{\beta}$ is multivariate normal. In particular,
\begin{equation}
\label{eq:normal_beta}
\bm{\beta}|\bm{\omega},\bm{y}  \sim N\left(\left(\bm{X}^{T}\bm{\Omega}\bm{X}\right)^{-1}\bm{X}^{T}\bm{\kappa},\left(\bm{X}^{T}\bm{\Omega}\bm{X}\right)^{-1}\right).
\end{equation}

\section{Geometric ergodicity of P\'olya-Gamma Gibbs sampler}
\label{sec:ge}
Let $\{\bm{\beta}^{(m)}, \bm{\omega}^{(m)} \}_{m=0}^{\infty}$ denote
the Markov chain associated with PS\&W's DA algorithm. In Bayesian
logistic regression models, inferences on $\bm{\beta}$ are made based on
the $\{\bm{\beta}^{(m)} \}_{m=0}^{\infty}$ sub-chain. As mentioned in
the introduction, the DA Markov chain is Harris ergodic. Let
$h:\mathbb{R}^p \rightarrow \mathbb{R}$ be a real valued function of $\bm{\beta}$ with
$\int_{\mathbb{R}^p} \vert h(\bm{\beta}) \vert \pi (\bm{\beta} \vert
\bm{y}) d\bm{\beta} < \infty $, then the posterior mean $E(h(\bm{\beta})|\bm{y})$ can be consistently
estimated by $\bar{h}_m = \sum_{i=0}^{m-1}h(\bm{\beta}^{(i)})/m$ for
any starting value $\bm{\beta}^{(0)}$ (see \ref{app_chen_condition}
for a discussion on the existence of finite moments for
\eqref{eq_marginal}).  We can build a CLT for $\bar{h}_m$ if there
exists a constant $\sigma^2_h \in (0,\infty)$ such that,
\begin{equation}
\label{eq:clt}
\sqrt{m}\left(\bar{h}_m - E(h(\bm{\beta})|\bm{y})\right)\overset{d}\rightarrow N\left(0,\sigma_h^2\right) \, \text{as } m\rightarrow \infty .
\end{equation}
Mere Harris ergodicity of the Markov chain $\{\bm{\beta}^{(m)}\}_{m=0}^{\infty}$  does not ensure that the CLT in
\eqref{eq:clt} holds. It turns out that the geometric rate of convergence
defined below guarantees the CLT under a finite second moment
condition \citep{roberts1997geometric}. Also it turns out that all the three
Markov chains
$\{\bm{\beta}^{(m)}, \bm{\omega}^{(m)} \}_{m=0}^{\infty}$,
$\{\bm{\beta}^{(m)}\}_{m=0}^{\infty}$ and
$\{\bm{\omega}^{(m)} \}_{m=0}^{\infty}$ have the same rate of
convergence \citep{robe:rose:2001}. Thus geometric ergodicity is a solidarity property of these three Markov chains. In this article we analyze the
$\{\bm{\omega}^{(m)} \}_{m=0}^{\infty}$ sub-chain denoted as
$\bm{\Psi}=\{ \bm{\omega}^{\left(m\right)}\}
_{m=0}^{\infty}$.  Let $\bm{\omega}^{\prime}$ be the current state and
$\bm{\omega}$ be the next state, then the Markov transition density
(Mtd) of $\bm{\Psi}$ is
\begin{equation}
\label{eq:mtd}
k\left(\bm{\omega}|\bm{\omega}^{\prime}\right)=\int_{\mathbb{R}^{p}}\pi\left(\bm{\omega}|\bm{\beta},\bm{y}\right)\pi\left(\bm{\beta}|\bm{\omega}^{\prime},\bm{y}\right)d\bm{\beta},
\end{equation}
where $\pi(\cdot|\cdot,\bm{y})$'s are the conditional densities
defined in \eqref{eq:cond_omega} and \eqref{eq:cond_beta}. Note that the Mtd of the $\{\bm{\beta}^{(m)}\}_{m=0}^{\infty}$  sub-chain (that is, when $\bm{\omega}$ is updated first) is similarly given by
\begin{equation*}
\label{eq:mtd_beta}
\tilde{k}\left(\bm{\beta}|\bm{\beta}^{\prime}\right)=\int_{\mathbb{R}_+^{n}}\pi\left(\bm{\beta}|\bm{\omega},\bm{y}\right)\pi\left(\bm{\omega}|\bm{\beta}^{\prime},\bm{y}\right)d\bm{\omega}.
\end{equation*}
 Let $\mathscr{B}$ denote the Borel $\sigma$-algebra of
$\mathbb{R}_{+}^{n}$ and $K(\cdot,\cdot)$ be the Markov transition
function corresponding to the Mtd $k(\cdot \vert \cdot)$ in
\eqref{eq:mtd}, that is, for any set $A \in \mathscr{B}$,
$\bm{\omega}^{\prime} \in \mathbb{R}_+^{n}$ and any $j=0,1,\dots,$
\begin{equation}
K(\bm{\omega}^{\prime}, A) = \mbox{Pr}(\bm{\omega}^{(j+1)} \in A | \bm{\omega}^{(j)} = \bm{\omega}^{\prime}) = \int_{A} k(\bm{\omega}| \bm{\omega}^\prime) d\bm{\omega}.
\end{equation}
Then the $m$-step Markov transition function is
$K^m(\bm{\omega}^{\prime}, A) = \mbox{Pr}(\bm{\omega}^{(m+j)} \in A |
\bm{\omega}^{(j)} = \bm{\omega}^{\prime})$. Let $\Pi(\cdot|\bm{y})$ be
the probability measure with density $\pi(\bm{\omega}|\bm{y})$, where
$\pi(\bm{\omega}|\bm{y}) = \int_{\mathbb{R}^p}
\pi(\bm{\beta},\bm{\omega}|\bm{y}) d\bm{\beta}$ and
$\pi(\bm{\beta},\bm{\omega}\vert \bm{y})$ is the joint density defined
in \eqref{eq_joint1}. The Markov chain $\bm{\Psi}$ is geometrically
ergodic if there exists a constant $0 < t <1$ and a function
$H: \mathbb{R}_+^{n} \mapsto [0,\infty)$ such that for any
$\bm{\omega} \in \mathbb{R}_+^{n}$, and $m \ge 1$,
\begin{equation}
\label{eq:ge}
||K^m(\bm{\omega},\cdot) - \Pi(\cdot|\bm{y})||_{ \mbox{{\small TV}}}:=\sup_{A\in \mathscr{B}} |K^m(\bm{\omega}, A) - \Pi(A|\bm{y})| \leq H(\bm{\omega}) t^m.
\end{equation}
Harris ergodicity of $\bm{\Psi}$ implies that
$||K^m(\bm{\omega},\cdot) - \Pi(\cdot|\bm{y})||_{ \mbox{{\small TV}}} \downarrow 0$ as
$m \rightarrow \infty$, while \eqref{eq:ge} guarantees its exponential
rate of convergence. Since the Markov chains
$\{ \bm{\beta}^{(m)} \}_{m=0}^{\infty}$ and
$\{ \bm{\omega}^{(m)} \}_{m=0}^{\infty}$ have the same rate of
convergence, \eqref{eq:ge} implies
$\{ \bm{\beta}^{(m)} \}_{m=0}^{\infty}$ is geometrically ergodic.
\cite{roberts1997geometric} show that since $\{ \bm{\beta}^{(m)} \}_{m=0}^{\infty}$ is reversible,
if \eqref{eq:ge} holds then there exists a CLT, that is, for any
$h:\mathbb{R}^p\rightarrow \mathbb{R}$ with
$E[ h(\bm{\beta})^2\vert \bm{y}] < \infty $, \eqref{eq:clt}
holds. Also, under \eqref{eq:ge} a consistent estimator of
$\sigma_h^2$ can be found by batch means or spectral variance methods
\citep{flegal2010batch}. The following theorem shows that the Markov chain
$\bm{\Psi}$ converges at a geometric rate.

\begin{theorem}
\label{them_logistic}
If the posterior density $\pi(\bm{\beta}\vert \bm{y})$ given in \eqref{eq_marginal} is proper, the Markov chain $\bm{\Psi}$ is geometrically ergodic.
\end{theorem}

\begin{remark}
  The conditions in Theorem \ref{them_logistic} are the same as the
  necessary and sufficient conditions for posterior propriety given in
  \ref{app_chen_condition}. Besides these two conditions, geometric
  ergodicity of $\bm{\Psi}$ does not need any other conditions.
\end{remark}


\begin{proof}[Proof of Theorem~\ref{them_logistic}]
We prove geometric ergodicity of $\bm{\Psi}$ by establishing a drift condition. In particular, we consider the drift function
\begin{equation}
\label{eq_drift}
V\left(\bm{\omega}\right)=\alpha\sum_{i=1}^{n}\frac{1}{\omega_{i}}+\sum_{i=1}^{n}\frac{1}{\sqrt{\omega_{i}}}+\sum_{i=1}^{n}\omega_{i},
\end{equation}
where $\alpha$ is a positive constant and show that for any $\bm{\omega},\bm{\omega}^{\prime} \in \mathbb{R}_+^n$, there exist some constants $\rho\in\left(0,1\right)$ and $L_0>0$ such that
\begin{equation}
\label{eq_drift_logistic}
E\left[V\left(\bm{\omega}\right)|\bm{\omega}^{\prime}\right]\leq\rho V\left(\bm{\omega}^{\prime}\right)+L_0.
\end{equation}
In \eqref{eq_drift_logistic} the expectation is with respect to the
Mtd $k(\bm{\omega} \vert \bm{\omega}^{\prime})$ defined in \eqref{eq:mtd}. Note that $V(\bm{\omega})$ is unbounded
off compact sets, that is, for any $a>0$, the set
$\{\bm{\omega}: V(\bm{\omega}) \leq a\}$ is compact. We now show that
$\bm{\omega}$-chain is a Feller chain, which means
$K\left(\bm{\omega},O\right)$ is a lower semi-continuous function on
$\mathbb{R}_{+}^{n}$ for each fixed open set $O$. Consider a sequence $\bm{\omega}_m$ with
$\bm{\omega}_m\rightarrow \bm{\omega}$ as $m \rightarrow \infty$. Note that,
\begin{align*}
\liminf_{m\rightarrow\infty} K\left(\bm{\omega}_{m},O\right) & =\liminf_{m\rightarrow\infty}\int_{O}k\left(\bm{\omega}|\bm{\omega}_{m}\right)d\bm{\omega}\\
 & =\liminf_{m\rightarrow\infty}\int_{O}\left[\int_{\mathbb{R}^{p}}\pi(\bm{\omega}|\bm{\beta},\bm{y})\pi(\bm{\beta}|\bm{\omega}_{m},\bm{y})d\bm{\beta}\right]d\bm{\omega}\\
 & \geq\int_{O}\int_{\mathbb{R}^{p}}\pi(\bm{\omega}|\bm{\beta},\bm{y})\liminf_{m\rightarrow\infty}\pi(\bm{\beta}|\bm{\omega}_{m},\bm{y})d\bm{\beta}d\bm{\omega},
\end{align*}
where the inequality follows from Fatou's lemma. Since $\pi(\bm{\beta}|\bm{\omega},\bm{y})$ is a continuous function in $\bm{\omega}$ and $\bm{\omega}_m\rightarrow \bm{\omega}$,
\begin{align*}
\liminf_{m\rightarrow\infty} K\left(\bm{\omega}_{m},O\right) & \geq\int_{O}\int_{\mathbb{R}^{p}}\pi(\bm{\omega}|\bm{\beta},\bm{y})\pi(\bm{\beta}|\bm{\omega},\bm{y})d\bm{\beta}d\bm{\omega}\\
 & =K\left(\bm{\omega},O\right).
\end{align*}
Thus by \cite{meyn1993markov}(chap. 15), \eqref{eq_drift_logistic}
implies that the Markov chain $\bm{\Psi}$ is geometrically ergodic.

Now we establish \eqref{eq_drift_logistic}. From the definition of the Mtd of $\bm{\Psi}$ in \eqref{eq:mtd}, it follows that
\begin{equation}
\label{eq:twosteps}
E\left[V\left(\bm{\omega}\right)|\bm{\omega}^{\prime}\right] = E \left\lbrace E\left[V\left(\bm{\omega}\right)\vert \bm{\beta},\bm{y}\right] \vert \bm{\omega}^{\prime} , \bm{y}\right\rbrace,
\end{equation}
where $E\left[ \cdot \vert \bm{\beta}, \bm{y} \right]$ denotes the expectation with respect to $\pi(\cdot \vert \bm{\beta},\bm{y})$ given in \eqref{eq:cond_omega} and $E\left\{ \cdot \vert \bm{\omega}^{\prime}, \bm{y} \right\}$ denotes the expectation with respect to $\pi(\cdot \vert \bm{\omega}^{\prime},\bm{y})$ given in \eqref{eq:cond_beta} .

We first evaluate the inner expectation in \eqref{eq:twosteps}, that is the expectation of $V\left(\bm{\omega}\right)$ with respect to $\pi\left(\bm{\omega}|\bm{\beta},\bm{y}\right)$. From \eqref{eq:pg_omega}, we know that $\omega_i\vert \bm{\beta},\bm{y} \sim PG\left( 1, \vert \bm{x}_i^T\bm{\beta} \vert \right)$. Thus by Lemma \ref{lem_exp} and Lemma \ref{lem_expinv} given in \ref{app_pg_props}, we have
\[
E\left(\omega_{i}|\bm{\beta},\bm{y}\right)=\frac{1}{2\left|\bm{x}_{i}^{T}\bm{\beta}\right|}\frac{\exp\left(\left|\bm{x}_{i}^{T}\bm{\beta}\right|\right)-1}{\exp\left(\left|\bm{x}_{i}^{T}\bm{\beta}\right|\right)+1}\leq\frac{1}{4},
\]

\[
E\left(\frac{1}{\omega_{i}}\mid\bm{\beta},\bm{y}\right)\leq2\left|\bm{x}_{i}^{T}\bm{\beta}\right|+L_{1}, \text{ and}
\]

\[
E\left(\frac{1}{\sqrt{\omega_{i}}}\mid\bm{\beta},\bm{y}\right)\leq\sqrt{2}\left|\bm{x}_{i}^{T}\bm{\beta}\right|^{1/2}+L_{2},
\]
where $L_1\equiv L(1)$,  $L_2 \equiv L(1/2)$ and $L(\cdot)$ is a function defined in Lemma \ref{lem_expinv}. Then
\begin{equation}
\label{eq:first_step}
E\left[V\left(\bm{\omega}\right)\mid\bm{\beta},\bm{y}\right]\leq2\alpha\sum_{i=1}^{n}\left|\bm{x}_{i}^{T}\bm{\beta}\right|+\sqrt{2}\sum_{i=1}^{n}\left|\bm{x}_{i}^{T}\bm{\beta}\right|^{1/2}+\alpha nL_{1}+nL_{2}+\frac{n}{4}.
\end{equation}

Now we consider the outer expectation in \eqref{eq:twosteps}, that is, the expectation with respect to $\pi(\bm{\beta}\vert \bm{\omega}^{\prime},\bm{y})$. Let
\[
\mu_{i}=\bm{x}_{i}^{T}\left(\bm{X}^{T}\bm{\Omega}^{\prime}\bm{X}\right)^{-1}\bm{X}^{T}\bm{\kappa},
\]
and \[
\sigma_{i}^{2}=\bm{x}_{i}^{T}\left(\bm{X}^{T}\bm{\Omega}^{\prime}\bm{X}\right)^{-1}\bm{x_{i}},
\]
where $\bm{\Omega}^{\prime}$ is the diagonal matrix with elements
$\omega^{\prime}_i$'s. From \eqref{eq:normal_beta} we know that
$\bm{x}_{i}^{T}\bm{\beta} \vert \bm{\omega}^{\prime}, \bm{y}\sim
N\left(\mu_{i},\sigma_{i}^{2}\right)$.  Then
$\left|\bm{x}_{i}^{T}\bm{\beta}\right|$ has a folded normal
distribution. Let $G(\cdot)$ denote the cumulative distribution function of the
standard normal random variable. So
\begin{equation}
\label{eq:folded_normal}
E\left(\left|\bm{x}_{i}^{T}\bm{\beta}\right|\mid\bm{\omega}^{\prime},\bm{y}\right)=\sigma_{i}\sqrt{\frac{2}{\pi}}e^{-\mu_{i}^{2}/2\sigma_{i}^{2}}+\mu_{i}\left(1-2G\left(-\frac{\mu_{i}}{\sigma_{i}}\right)\right)\leq\sigma_{i}\sqrt{\frac{2}{\pi}}+\left|\mu_{i}\right|.
\end{equation}

By the inequality in \cite{roy2010monte} [Lemma 3],
\begin{equation}
\label{eq:sigma_bound}
\sigma_{i}^{2}=\bm{x}_{i}^{T}\left(\omega_{i}^{\prime}\bm{x}_{i}\bm{x}_{i}^{T}+\sum_{j\neq i}\omega_{j}^{\prime}\bm{x}_{j}\bm{x}_{j}^{T}\right)^{-1}\bm{x}_{i} = \frac{1}{\omega_i^{\prime}}\bm{x}_{i}^{T}\left(\bm{x}_{i}\bm{x}_{i}^{T}+\sum_{j\neq i}\frac{\omega_{j}^{\prime}}{\omega_{i}^{\prime}}\bm{x}_{j}\bm{x}_{j}^{T}\right)^{-1}\bm{x}_{i}  \leq\frac{1}{\omega_{i}^{\prime}}.
\end{equation}
Also, \[
\sum_{i=1}^{n}\left|\mu_{i}\right|=\sum_{i=1}^{n}\left|\bm{x}_{i}^{T}\left(\bm{X}^{T}\bm{\Omega}^{\prime}\bm{X}\right)^{-1}\bm{X}^{T}\bm{\kappa}\right|=\bm{l}^{T}\bm{X}\left(\bm{X}^{T}\bm{\Omega}^{\prime}\bm{X}\right)^{-1}\bm{X}^{T}\bm{\kappa},
\]
where $\bm{l} = (l_1,\dots,l_n)$ with $l_{i}=1$ if $\mu_{i}\geq0$ and $l_{i}=-1$ if $\mu_{i}<0$. By Cauchy-Schwarz inequality, we have
\begin{align}
\label{eq:sum_mu}
\sum_{i=1}^{n}\left|\mu_{i}\right| & =\left|\bm{l}^{T}\bm{X}\left(\bm{X}^{T}\bm{\Omega}^{\prime}\bm{X}\right)^{-1/2}\left(\bm{X}^{T}\bm{\Omega}^{\prime}\bm{X}\right)^{-1/2}\bm{X}^{T}\bm{\kappa}\right| \nonumber\\
 & \leq\sqrt{\bm{l}^{T}\bm{X}\left(\bm{X}^{T}\bm{\Omega}^{\prime}\bm{X}\right)^{-1}\bm{X}^{T}\bm{l}}\sqrt{\bm{\kappa}^{T}\bm{X}\left(\bm{X}^{T}\bm{\Omega}^{\prime}\bm{X}\right)^{-1}\bm{X}^{T}\bm{\kappa}} \,.
\end{align}
Now \begin{align}
\label{eq:quad_l}
\bm{l}^{T}\bm{X}\left(\bm{X}^{T}\bm{\Omega}^{\prime}\bm{X}\right)^{-1}\bm{X}^{T}\bm{l} & =\bm{l}^{T}\left(\bm{\Omega}^{\prime}\right)^{-1/2}\left(\bm{\Omega}^{\prime}\right)^{1/2}\bm{X}\left(\bm{X}^{T}\bm{\Omega}^{\prime}\bm{X}\right)^{-1}\bm{X}^{T}\left(\bm{\Omega}^{\prime}\right)^{1/2}\left(\bm{\Omega}^{\prime}\right)^{-1/2}\bm{l} \nonumber \\
 & \leq\bm{l}^{T}\left(\bm{\Omega}^{\prime}\right)^{-1}\bm{l} = \sum_{i=1}^n\frac{1}{\omega_{i}^{\prime}},
\end{align}
where the inequality follows from the fact that $\bm{I} - \left(\bm{\Omega}^{\prime}\right)^{1/2}\bm{X}\left(\bm{X}^{T}\bm{\Omega}^{\prime}\bm{X}\right)^{-1}\bm{X}^{T}\left(\bm{\Omega}^{\prime}\right)^{1/2}$ is a positive semidefinite matrix.

Since the posterior density \eqref{eq_marginal} is assumed proper, the two conditions of Proposition \ref{prop1} given in \ref{app_chen_condition} hold. Thus by Lemma \ref{lem_rho} presented in \ref{app_matrix} and the facts  $\bm{x}_i\bm{x}_i^T = \bm{z}_i \bm{z}_i^T$, $\kappa_i \bm{x}_i = -(1/2)\bm{z}_i$, there exists a constant $\rho_1 \in (0,1)$ such that
\begin{equation}
\label{eq:rho1_ineq}
\bm{\kappa}^{T}\bm{X}\left(\bm{X}^{T}\bm{\Omega}^{\prime}\bm{X}\right)^{-1}\bm{X}^{T}\bm{\kappa}=\frac{1}{4}\bm{1}^{T}\bm{Z}\left(\bm{Z}^{T}\bm{\Omega}^{\prime}\bm{Z}\right)^{-1}\bm{Z}^{T}\bm{1}\leq\frac{1}{4}\rho_{1}\sum_{i=1}^{n}\frac{1}{\omega_{i}^{\prime}},
\end{equation}
where $\bm{Z}$ is defined in \ref{app_chen_condition}, and $\bm{1}$
is the $n\times 1$ vector of $1$'s.

Using \eqref{eq:sigma_bound} - \eqref{eq:rho1_ineq}, from \eqref{eq:folded_normal} we have
\begin{equation}
\label{eq:exp_bound}
E\left(\sum_{i=1}^{n}\left|\bm{x}_{i}^{T}\bm{\beta}\right||\bm{\omega}^{\prime},\bm{y}\right)\leq\frac{1}{2}\sqrt{\rho_{1}}\sum_{i=1}^{n}\frac{1}{\omega_{i}^{\prime}}+\sqrt{\frac{2}{\pi}}\sum_{i=1}^{n}\frac{1}{\sqrt{\omega_{i}^{\prime}}}.
\end{equation}
Using the inequality  $2ab \leq a^2 + b^2$, we have for any $c_1 >0$,
\begin{equation}
\label{eq:sqrt_bound1}
E\left(\sqrt{2}\left|\bm{x}_{i}^{T}\bm{\beta}\right|^{1/2}|\bm{\omega}^{\prime},\bm{y}\right)=E\left(2\frac{\sqrt{2}}{2c_{1}}c_{1}\left|\bm{x}_{i}^{T}\bm{\beta}\right|^{1/2}|\bm{\omega}^{\prime},\bm{y}\right)\leq c_{1}^{2}E\left(\left|\bm{x}_{i}^{T}\bm{\beta}\right|\vert\bm{\omega}^{\prime},\bm{y}\right)+\frac{1}{2c_{1}^{2}}.
\end{equation}
Using \eqref{eq:exp_bound} and \eqref{eq:sqrt_bound1}, we have
\begin{align}
\label{eq:sqrt_bound2}
E\left(\sqrt{2}\sum_{i=1}^{n}\left|\bm{x}_{i}^{T}\bm{\beta}\right|^{1/2}|\bm{\omega}^{\prime},\bm{y}\right) & \leq c_{1}^{2}\sum_{i=1}^{n}E\left(\left|\bm{x}_{i}^{T}\bm{\beta}\right|\vert\bm{\omega}^{\prime},\bm{y}\right)+\frac{n}{2c_{1}^{2}} \nonumber \\
 & \leq\frac{1}{2}c_{1}^{2}\sqrt{\rho_{1}}\sum_{i=1}^{n}\frac{1}{\omega_{i}^{\prime}}+c_{1}^{2}\sqrt{\frac{2}{\pi}}\sum_{i=1}^{n}\frac{1}{\sqrt{\omega_{i}^{\prime}}}+\frac{n}{2c_{1}^{2}}.
\end{align}
Combining \eqref{eq:first_step}, \eqref{eq:exp_bound} and \eqref{eq:sqrt_bound2},  from \eqref{eq:twosteps} we have
\[
E\left[V\left(\bm{\omega}\right)\mid\bm{\omega}^{\prime}\right]\leq\alpha\sqrt{\rho_{1}}\left(1+\frac{c_{1}^{2}}{2\alpha}\right)\sum_{i=1}^{n}\frac{1}{\omega_{i}^{\prime}}+\sqrt{\frac{2}{\pi}}\left(2\alpha+c_{1}^{2}\right)\sum_{i=1}^{n}\frac{1}{\sqrt{\omega_{i}^{\prime}}}+\frac{n}{2c_{1}^{2}}+\alpha nL_{1}+nL_{2}+\frac{n}{4}.
\]
We now show that there exist $c_1$ and $\alpha$ such that $\sqrt{\rho_{1}}\left(1+c_{1}^{2}/\left(2\alpha\right)\right)<1$ and $\sqrt{2/\pi}\left(2\alpha+c_{1}^{2}\right)<1$, that is
\begin{equation}
\label{eq:ineq_coef}
\frac{c_{1}^{2}}{2}\frac{\sqrt{\rho_{1}}}{1-\sqrt{\rho_{1}}}<\alpha<\frac{1}{2}\left(\sqrt{\frac{\pi}{2}}-c_{1}^{2}\right).
\end{equation}
So we need to show there exists $c_1$ such that $\sqrt{\frac{\pi}{2}}-c_{1}^{2}>c_{1}^{2}\sqrt{\rho_{1}}/\left(1-\sqrt{\rho_{1}}\right)$. Thus for any $c_1$ with $c_{1}^{2}<\sqrt{\pi/2}\left(1-\sqrt{\rho_{1}}\right)$, we can choose $\alpha$ satisfying \eqref{eq:ineq_coef}.  So there exist $c_{1}$ and $\alpha$ such that
\[
E\left[V\left(\bm{\omega}\right)|\bm{\omega}^{\prime}\right]\leq\rho V\left(\bm{\omega}^{\prime}\right)+L_0,
\]
where
\begin{align*}
\rho & =\max\left\{ \sqrt{\rho_{1}}\left(1+\frac{c_{1}^{2}}{2\alpha}\right),\sqrt{\frac{2}{\pi}}\left(2\alpha+c_{1}^{2}\right)\right\} <1,\\
L_0 & =\frac{n}{2c_{1}^{2}}+\alpha nL_{1}+nL_{2}+\frac{n}{4}.
\end{align*}
\end{proof}

\section{Summary}
\label{sec:disc}
In this article, we prove the geometric rate of convergence for the
\pcite{polson2013bayesian} P{\'o}lya-Gamma Gibbs sampler for the Bayesian
logistic regression with a flat prior on the regression coefficients
$\bm{\beta}$. The conditions for geometric ergodicity are the same as
the necessary and sufficient conditions for posterior propriety. That
means, the Gibbs sampler is always geometrically ergodic if the
posterior distribution is proper. If the posterior is improper, the Gibbs sampler is not even positive recurrent and the usual sample average estimator is inconsistent for the posterior mean \citep{athreya2014monte}. Thus our result guarantees availability of a CLT for the time average estimator as long as it is consistent.
  \cite{roy2007convergence} established a similar
result for \pcite{albert1993bayesian} DA algorithm for the Bayesian probit
regression model with a flat prior on $\bm{\beta}$. The latent
variables in \pcite{albert1993bayesian} DA algorithm are normal random
variables and their conditional (posterior) distributions are
truncated normal. Since the latent variables in
\pcite{polson2013bayesian} DA algorithm have the non-standard PG
distribution, it turns out the drift function, inequalities,
techniques used in our proof are quite different from those of
\cite{roy2007convergence}.  One potential future work is to study the
convergence properties of the P{\'o}lya-Gamma Gibbs sampler for Bayesian
logistic mixed models under improper priors for both regression
coefficients and variance components.

\section*{Appendix}
\renewcommand{\thesubsection}{A.\arabic{subsection}}

\subsection{\pcite{chen2001propriety} conditions for posterior propriety}
\label{app_chen_condition}

Let $\bm{X}$ denote the $n\times p$ design matrix with the $i$th row
$\bm{x}_i^T$ and $\bm{Z}$ be the $n\times p$ matrix with the $i$th row
$\bm{z}_i^T = c_i\bm{x}_i^T $, where $c_i = 1$ if $y_i = 0$ and
$c_i = -1$ if $y_i = 1$ for $i=1,\dots,n$. The following proposition
gives the necessary and sufficient conditions for propriety of the
posterior density \eqref{eq_marginal}.

\begin{proposition}
\label{prop1}
\citep{chen2001propriety}. The marginal density $c(\bm{y})$ is finite if and only if
\begin{enumerate}
\item $\bm{X}$ is a full rank matrix;
\item There exists a vector $\bm{e} = \left(e_1,\dots,e_n\right)^T$ with strictly positive components such that $\bm{Z}^T\bm{e} = \bm{0}_p$.
\end{enumerate}
\end{proposition}

\begin{remark}
\label{rem1_logistic}
\cite{roy2007convergence} provide a method for checking the second condition in Proposition~\ref{prop1}. This method can be easily implemented using publicly available software packages.
\end{remark}

\begin{remark}
\label{remark_app}
Since the moment generating function of the logistic distribution
exists from \cite{chen2001propriety}[Theorem 2.3], it follows that
under the two conditions of Proposition \ref{prop1},
$\int_{\mathbb{R}^p} e^{\delta \Vert \bm{\beta} \Vert} \pi (\bm{\beta
  \vert \bm{y}}) d\bm{\beta} < \infty$ for some $\delta >0$ and
$\int_{\mathbb{R}^p} \Vert \bm{\beta} \Vert^r \pi(\bm{\beta} \vert
\bm{y}) d\bm{y} <\infty$ for all $r\geq 0$.
\end{remark}

\subsection{Some useful properties of P\'olya-Gamma distribution}
\label{app_pg_props}

\begin{lemma}
\label{lem_exp}
If $\omega\sim\text{PG}\left(1,b\right)$, $E\left(\omega\right)\leq\frac{1}{4}$.
\end{lemma}
\begin{proof}
From \cite{polson2013bayesian},  we know that
\[
E\left(\omega\right)=\frac{1}{2b}\frac{e^{b}-1}{e^{b}+1}.
\]
Consider the function $f\left(x\right)=(e^{x}-1)/\left[x(e^{x}+1)\right]$, then
\[
f^{\prime}\left(x\right)=\frac{2xe^{x}-e^{2x}+1}{\left[x\left(e^{x}+1\right)\right]^{2}}.
\]
Consider another function $f_{1}\left(x\right)=2xe^{x}-e^{2x}+1$. We
have $f_{1}^{\prime}\left(x\right)=2e^{x}\left(1+x-e^{x}\right)$. We know that $1+x-e^{x}\leq0$ for $x\geq0$. So
$f_{1}^{\prime}\left(x\right)\leq0$ for $x\geq0$. Hence
$f_{1}\left(x\right)\leq f_{1}\left(0\right)=0$. Therefore,
$f^{\prime}\left(x\right)\leq0$ for $x\geq 0$. Then for $x\geq 0$,
$f\left(x\right)\leq \lim_{x\rightarrow0} f(x)=1/2$. So
$E\left(\omega\right)\leq1/4$.
\end{proof}

\begin{lemma}
\label{lem_expinv}
If $\omega\sim\text{PG}\left(1,b\right)$,  for $0<s\leq1$,
\[
E\left(\omega^{-s}\right)\leq2^{s}b^s+L\left(s\right),
\]
 where $L\left(s\right)$ is a constant depending on $s$.
\end{lemma}
\begin{proof}

From \eqref{eq_polya}, the probability density function of $\text{PG}\left(1,b\right)$ is,
\[
f\left(x|1,b\right)=\cosh\left(b/2\right)\sum_{n=0}^{\infty}\left(-1\right)^{n}\frac{\left(2n+1\right)}{\sqrt{2\text{\ensuremath{\pi}}x^{3}}}e^{-\frac{\left(2n+1\right)^{2}}{8x}-\frac{b^{2}}{2}x}.
\]
We consider the two cases, $ b = 0$ and $b\neq 0$ separately.

\noindent \textbf{Case 1}: $ b = 0$.
Since $0 < s \leq 1$, for any $x >0$, $x^{-s} \leq x^{-1} + 1$. Thus
\begin{align*}
E\left(\omega^{-s}\right) & \leq\int_{0}^{\infty}\left(x^{-1}+1\right)f\left(x|1,0\right)dx=\int_{0}^{\infty}x^{-1}f\left(x|1,0\right)dx+1.
\end{align*}
Now
\begin{align*}
\int_{0}^{\infty}x^{-1}f\left(x|1,0\right)dx & =\sum_{n=0}^{\infty}\left(-1\right)^{n}\frac{\left(2n+1\right)}{\sqrt{2\pi}}\int_{0}^{\infty}x^{-5/2}e^{-\frac{\left(2n+1\right)^{2}}{8x}}dx\\
 & =2^{3}\sum_{n=0}^{\infty}\left(-1\right)^{n}\frac{1}{\left(2n+1\right)^{2}}=8C,
\end{align*}
where $C$ is Catalan's constant. Hence $E\left(\omega^{-s}\right)\leq8C+1$.

\noindent\textbf{Case 2}:  $b\neq 0$.
Note that,
\begin{align*}
E\left(\omega^{-s}\right) & =\int_{0}^{\infty}x^{-s}f\left(x|1,b\right)dx\\
 & =\cosh\left(b/2\right)\sum_{n=0}^{\infty}\left(-1\right)^{n}\frac{\left(2n+1\right)}{\sqrt{2\pi}}\int_{0}^{\infty}\frac{1}{\sqrt{x^{3}}}x^{-s}e^{-\frac{\left(2n+1\right)^{2}}{8x}-\frac{b^{2}}{2}x}dx.
\end{align*}
According to \cite{olver2010math}[10.32.10], we have \begin{align*}
\int_{0}^{\infty}\frac{1}{\sqrt{x^{3}}}x^{-s}e^{-\frac{\left(2n+1\right)^{2}}{8x}-\frac{b^{2}}{2}x}dx & =\int_{0}^{\infty}x^{-s-\frac{3}{2}}e^{-\frac{\left(2n+1\right)^{2}}{8x}-\frac{b^{2}}{2}x}dx\\
 & =2K_{s+\frac{1}{2}}\left(\frac{b\left(2n+1\right)}{2}\right)\cdot\left(\frac{2b}{2n+1}\right)^{s+\frac{1}{2}},
\end{align*}
where $K_{\nu} (x)$ is the modified Bessel function of the second kind. For $x>0$, according to \cite{olver2010math}[10.32.8],
\begin{align*}
 K_{s+\frac{1}{2}}\left(x\right) & =\frac{\sqrt{\pi}\left(\frac{1}{2}x\right)^{s+1/2}}{\Gamma\left(s+1\right)}\int_{1}^{\infty}e^{-xt}\left(t^{2}-1\right)^{s}dt\\
 & =\frac{\sqrt{\pi}\left(\frac{1}{2}x\right)^{s+1/2}}{\Gamma\left(s+1\right)}e^{-x}\int_{0}^{\infty}e^{-xt}\left(t^{2}+2t\right)^{s}dt\\
 & \leq\frac{\sqrt{\pi}\left(\frac{1}{2}x\right)^{s+1/2}}{\Gamma\left(s+1\right)}e^{-x}\int_{0}^{\infty}e^{-xt}\left(t^{2s}+2^{s}t^{s}\right)dt\\
 & =\frac{\sqrt{\pi}\left(\frac{1}{2}x\right)^{s+1/2}}{\Gamma\left(s+1\right)}e^{-x}\left(\frac{\Gamma\left(2s+1\right)}{x^{2s+1}}+2^{s}\frac{\Gamma\left(s+1\right)}{x^{s+1}}\right)\\
 & =\sqrt{\pi}e^{-x}\left[\frac{\Gamma\left(2s+1\right)}{\Gamma\left(s+1\right)}2^{-s-1/2}x^{-s-1/2}+ 2^{-1/2}x^{-1/2}\right].
\end{align*}
Thus
\begin{align*}
2K_{s+\frac{1}{2}}\left(\frac{b\left(2n+1\right)}{2}\right)\cdot\left(\frac{2b}{2n+1}\right)^{s+\frac{1}{2}} & \leq2\sqrt{\pi}\exp\left(-nb-b/2\right)\\
 & \left[\frac{\Gamma\left(2s+1\right)}{\Gamma\left(s+1\right)}2^{s+1/2}\frac{1}{\left(2n+1\right)^{2s+1}}+2^{s+1/2}\frac{b^{s}}{\left(2n+1\right)^{s+1}}\right].
\end{align*}
Recall that $\cosh(b/2) = (e^{b/2} + e^{-b/2})/2$. Thus
\[
E\left(\frac{1}{\omega^{s}}\right)\leq\frac{1+e^{-b}}{2}\sum_{n=0}^{\infty}\left(-1\right)^{n}e^{-nb}\left[\frac{\Gamma\left(2s+1\right)}{\Gamma\left(s+1\right)}2^{s+1}\frac{1}{\left(2n+1\right)^{2s}}+2^{s+1}\frac{b^{s}}{\left(2n+1\right)^{s}}\right].
\]
Also,
\[
\sum_{n=0}^{\infty}\left(-e^{-b}\right)^{n}\frac{1}{\left(2n+1\right)^{s}}=2^{-s}\Phi\left(-e^{-b},s,\frac{1}{2}\right)=2^{-s}\frac{1}{\Gamma\left(s\right)}\int_{0}^{\infty}\frac{t^{s-1}e^{-\frac{1}{2}t}}{1+e^{-b-t}}dt,
\]
and
\begin{equation}
\label{eq:lerch}
\sum_{n=0}^{\infty}\left(-e^{-b}\right)^{n}\frac{1}{\left(2n+1\right)^{2s}}=2^{-2s}\Phi\left(-e^{-b},2s,\frac{1}{2}\right)=2^{-2s}\frac{1}{\Gamma\left(2s\right)}\int_{0}^{\infty}\frac{t^{2s-1}e^{-\frac{1}{2}t}}{1+e^{-b-t}}dt\leq1.
\end{equation}
where $\Phi(\cdot)$ is the Lerch transcendent function. The inequality
in \eqref{eq:lerch} follows from the fact that
$1 + e^{-b - t} \geq 1$. Thus we have,
\[
E\left(\frac{1}{\omega^{s}}\right)\leq\left(1+e^{-b}\right)\frac{b^{s}}{\Gamma\left(s\right)}\int_{0}^{\infty}\frac{t^{s-1}e^{-\frac{1}{2}t}}{1+e^{-b-t}}dt+2^{s+1}\frac{\Gamma\left(2s+1\right)}{\Gamma\left(s+1\right)}.
\]
For fixed $s>0$, let
\[f\left(b\right)
  \equiv\left(1+e^{-b}\right)\frac{b^{s}}{\Gamma\left(s\right)}\int_{0}^{\infty}\frac{t^{s-1}e^{-\frac{1}{2}t}}{1+e^{-b-t}}dt-2^{s}b^{s}.\]
Using the Dominated Convergence Theorem (DCT), we can show that
$f\left(b\right)$ is a continuous function of $b$. DCT can also be used to
show that $\lim_{b\rightarrow\infty}f\left(b\right)=0$ and
$f\left(0\right)=0$. So $\left|f\left(b\right)\right|$ can be bounded
by a positive constant value $f_{0}$. Thus we have
\[
E\left(\frac{1}{\omega^{s}}\right)\leq2^{s}b^{s}+2^{s+1}\frac{\Gamma\left(2s+1\right)}{\Gamma\left(s+1\right)}+f_{0}.
\]
Combining the two cases $b=0$ and $b\neq 0$, we have
\[
E\left(\frac{1}{\omega^{s}}\right)\leq2^{s}b^{s}+L\left(s\right),
\]
where $L\left(s\right)=\max\left\{ 2^{s+1}\frac{\Gamma\left(2s+1\right)}{\Gamma\left(s+1\right)}+f_{0},8C+1\right\}$.
\end{proof}

\subsection{A matrix result}
\label{app_matrix}

\begin{lemma}
\label{lem_rho}
For fixed
$\bm{\omega} = \left( \omega_1,\dots, \omega_n\right) \in
\mathbb{R}_{+}^{n}$, define $\bm{\Omega}$ to be the $n\times n$
diagonal matrix whose $i$th diagonal element is $\omega_i$. Let $\bm{1}$
be the $n\times 1$ vector of $1$'s. For a full rank $n\times p$ matrix
$\bm{Z}$, if there exists a positive $n\times1$ vector $\bm{e} = (e_1, e_2, \dots, e_n)$ such
that $\bm{Z}^{T}\bm{e}=\bm{0}$, then there exists a constant
$\rho_1 \in [0,1)$ such that
\[
\bm{1}^{T}\bm{Z}\left(\bm{Z}^{T}\bm{\Omega}\bm{Z}\right)^{-1}\bm{Z}^{T}\bm{1}\text{\ensuremath{\leq}}\rho_{1}\sum_{i=1}^{n}\frac{1}{\omega_{i}}.
\]
\end{lemma}

\begin{proof}
Let $\bm{\lambda}=\left(\lambda_{1,}\dots,\lambda_{n}\right)^{T} \in \mathbb{R}_+^n$,
where $\lambda_{i}=(1/\sqrt{\omega_{i}})/\sqrt{\sum_{i=1}^{n}(1/\omega_{i}})$,
and $\bm{\Lambda}=\text{diag}\left(\lambda_{1,},\dots,\lambda_{n}\right)$.
Define
\[S=\left\{ \bm{x}=\left(x_{1},\dots,x_{n}\right)^{T}:x_{i}\in\left(0,\infty\right) \text{for }i=1,\dots,n,\left\Vert \bm{x}\right\Vert =1\right\}, \]
and \\
\[S^{*}=\left\{ \bm{x}=\left(x_{1},\dots,x_{n}\right)^{T}:x_{i}\in\left[0,\infty\right)\text{for }i=1,\dots,n,\left\Vert \bm{x}\right\Vert =1\right\}. \]
The set $S^*$ is a compact set. Note that
\begin{equation}
  \label{eq:wlam}
\sup_{\bm{\omega}\in\mathbb{R}_{+}^{n}}\frac{\bm{1}^{T}\bm{Z}\left(\bm{Z}^{T}\bm{\Omega}\bm{Z}\right)^{-1}\bm{Z}^{T}\bm{1}}{\sum_{i=1}^n 1/\omega_{i}}=\sup_{\bm{\lambda}\in S}\bm{1}^{T}\bm{Z}\left(\bm{Z}^{T}\bm{\Lambda}^{-2}\bm{Z}\right)^{-1}\bm{Z}^{T}\bm{1}.  
\end{equation}
Now we study the supremum of
$\bm{1}^{T}\bm{Z}\left(\bm{Z}^{T}\bm{\Lambda}^{-2}\bm{Z}\right)^{-1}\bm{Z}^{T}\bm{1}$
over $\bm{\lambda} \in S$. We know that
$\bm{1}^{T}\bm{Z}\left(\bm{Z}^{T}\bm{\Lambda}^{-2}\bm{Z}\right)^{-1}\bm{Z}^{T}\bm{1}$
is a continuous function of $\bm{\lambda}$ in $S$. For
$\bm{\lambda}\in S^{*}\backslash S$, there exists a sequence
$\left\{ \bm{\lambda}_{m} \equiv (\lambda_{1,m},\dots, \lambda_{n,m})^T\in
  S\right\} _{m=1}^{\infty}$ such that
$\lim_{m\rightarrow\infty}\bm{\lambda}_{m}=\bm{\lambda}$. We define
the function $f(\cdot)$ on $S^{*}$ as
\[
f\left(\bm{\lambda}\right)\equiv\begin{cases}
\bm{1}^{T}\bm{Z}\left(\bm{Z}^{T}\bm{\Lambda}^{-2}\bm{Z}\right)^{-1}\bm{Z}^{T}\bm{1} & \bm{\lambda}\in S\\
\lim_{m\rightarrow\infty}\bm{1}^{T}\bm{Z}\left(\bm{Z}^{T}\bm{\Lambda}_{m}^{-2}\bm{Z}\right)^{-1}\bm{Z}^{T}\bm{1} & \bm{\lambda}\in S^{*}\backslash S,
\end{cases}
\]
where $\bm{\Lambda}_m = \text{diag}(\lambda_{1,m},\dots,\lambda_{n,m})$ and $\lim_{m\rightarrow\infty}\bm{\lambda}_{m}=\bm{\lambda}\in S^{*}\backslash S$ with $\bm{\lambda}_{m}\in S$. Then $f\left(\bm{\lambda}\right)$ is a continuous function on $S^{*}$. Also
\begin{equation}
\label{eq:supremum}
\sup_{\bm{\lambda}\in S}\bm{1}^{T}\bm{Z}\left(\bm{Z}^{T}\bm{\Lambda}^{-2}\bm{Z}\right)^{-1}\bm{Z}^{T}\bm{1}\leq\sup_{\bm{\lambda}\in S^{*}}f\left(\bm{\lambda}\right).
\end{equation}

We will now show that
$\sup_{\bm{\lambda}\in S^{*}}f\left(\bm{\lambda}\right) < 1$. First we
show that for any $\bm{\lambda} \in S$, $f(\bm{\lambda}) < 1$. Define
$\tilde{\bm{Z}} \equiv\bm{\Lambda}^{-1}\bm{Z}$, then
\begin{align}
\label{eq:trans}
\bm{1}^{T}\bm{Z}\left(\bm{Z}^{T}\bm{\Lambda}^{-2}\bm{Z}\right)^{-1}\bm{Z}^{T}\bm{1} & =\bm{1}^{T}\bm{\Lambda}\bm{\Lambda}^{-1}\bm{Z}\left(\bm{Z}^{T}\bm{\Lambda}^{-2}\bm{Z}\right)^{-1}\bm{Z}^{T}\bm{\Lambda}^{-1}\bm{\Lambda}\bm{1} \nonumber \\
 & =\bm{1}^{T}\bm{\Lambda}\tilde{\bm{Z}}\left(\tilde{\bm{Z}}^{T}\tilde{\bm{Z}}\right)^{-1}\tilde{\bm{Z}}^{T}\bm{\Lambda}\bm{1}=\bm{\lambda}^{T}\tilde{\bm{Z}}\left(\tilde{\bm{Z}}^{T}\tilde{\bm{Z}}\right)^{-1}\tilde{\bm{Z}}^{T}\bm{\lambda}.
\end{align}
Since by the assumption of Lemma \ref{lem_rho}, there exists a
positive vector $\bm{e}$ such that $\bm{Z}^T\bm{e} = \bm{0}$, we have
$\tilde{\bm{Z}}^T \bm{\Lambda} \bm{e} = \bm{Z}^T\bm{\Lambda}^{-1}
\bm{\Lambda}e = \bm{Z}^T \bm{e} = \bm{0}$.  Thus
$\tilde{\bm{Z}}\left(\tilde{\bm{Z}}^{T}\tilde{\bm{Z}}\right)^{-1}\tilde{\bm{Z}}^{T}\bm{\Lambda}
\bm{e} = \bm{0}$. In other words, $\bm{\Lambda}\bm{e}$ is an
eigenvector of
$\tilde{\bm{Z}}\left(\tilde{\bm{Z}}^{T}\tilde{\bm{Z}}\right)^{-1}\tilde{\bm{Z}}^{T}$
corresponding to eigenvalue zero. Since
$\bm{e}^T \bm{\Lambda} \bm{\lambda} = \sum_{i=1}^n \lambda_i^2e_i >0$,
and
$\tilde{\bm{Z}}\left(\tilde{\bm{Z}}^{T}\tilde{\bm{Z}}\right)^{-1}\tilde{\bm{Z}}^{T}$
is an idempotent matrix, it implies that $\bm{\lambda}$ cannot be an
eigenvector of
$\tilde{\bm{Z}}\left(\tilde{\bm{Z}}^{T}\tilde{\bm{Z}}\right)^{-1}\tilde{\bm{Z}}^{T}$
corresponding to eigenvalue 1 \citep[][Proposition 4.5.4]{bernstein2005matrix}. Thus
$\bm{\lambda}^T\tilde{\bm{Z}}\left(\tilde{\bm{Z}}^{T}\tilde{\bm{Z}}\right)^{-1}\tilde{\bm{Z}}^{T}
\bm{\lambda} < 1$, that is by \eqref{eq:trans}, $f(\bm{\lambda}) < 1$
for any $\bm{\lambda} \in S$.

Next we show that or any $\bm{\lambda} \in S^* \backslash S$,
$f(\bm{\lambda}) < 1$. Define
$\tilde{\bm{Z}}_m \equiv \bm{\Lambda}_m^{-1}\bm{Z}$. Now, we will show
that
$\lim_{m\rightarrow\infty}\tilde{\bm{Z}}_{m}\left(\tilde{\bm{Z}}_{m}^{T}\tilde{\bm{Z}}_{m}\right)^{-1}\tilde{\bm{Z}}_{m}^{T}$
exists. Define
$\bm{P}_{m}\equiv\tilde{\bm{Z}}_{m}\left(\tilde{\bm{Z}}_{m}^{T}\tilde{\bm{Z}}_{m}\right)^{-1}\tilde{\bm{Z}}_{m}^{T}$. We
will show that each element in $\bm{P}_{m}$ is bounded by 1. Let
$\bm{Z}\equiv\left(\bm{z}_{1},\dots,\bm{z}_{n}\right)^{T}$, then
$\tilde{\bm{Z}}_{m}=\left(\lambda_{1,m}^{-1}\bm{z}_{1},\dots,\lambda_{n,m}^{-1}\bm{z}_{n}\right)^{T}$. The
$\left(i,j\right)$th element of $\bm{P}_{m}$ is
$\lambda_{i,m}^{-1}\lambda_{j,m}^{-1}\bm{z}_{i}^{T}\left(\tilde{\bm{Z}}_{m}^{T}\tilde{\bm{Z}}_{m}\right)^{-1}\bm{z}_{j}$.
For $i=j$, using the inequality in \cite{roy2010monte} [Lemma 3], the $i$th
diagonal element of $\bm{P}_{m}$ is
\[
\lambda_{i,m}^{-2}\bm{z}_{i}^{T}\left(\tilde{\bm{Z}}_{m}^{T}\tilde{\bm{Z}}_{m}\right)^{-1}\bm{z}_{i}=\lambda_{i,m}^{-2}\bm{z}_{i}^{T}\left(\lambda_{i,m}^{-2}\bm{z}_{i}\bm{z}_{i}^{T}+\sum_{j=1,j\neq i}^{n}\lambda_{j,m}^{-2}\bm{z}_{j}\bm{z}_{j}^{T}\right)^{-1}\bm{z}_{i}\leq1.
\]
For $i\neq j$, by Cauchy\textendash Schwartz inequality
\[
\left|\lambda_{i,m}^{-1}\lambda_{j,m}^{-1}\bm{z}_{i}^{T}\left(\tilde{\bm{Z}}_{m}^{T}\tilde{\bm{Z}}_{m}\right)^{-1}\bm{z}_{j}\right|\leq\sqrt{\lambda_{i,m}^{-2}\bm{z}_{i}^{T}\left(\tilde{\bm{Z}}_{m}^{T}\tilde{\bm{Z}}_{m}\right)^{-1}\bm{z}_{i}}\sqrt{\lambda_{j,m}^{-2}\bm{z}_{j}^{T}\left(\tilde{\bm{Z}}_{m}^{T}\tilde{\bm{Z}}_{m}\right)^{-1}\bm{z}_{j}}\leq1.
\]
Since each element of $\bm{P}_m$ is a bounded, continuous function of
$\bm{\lambda}_m$ over $S$, its limit as $m\rightarrow \infty$ exists
and is bounded. Thus,
$\lim_{m\rightarrow\infty}\tilde{\bm{Z}}_{m}\left(\tilde{\bm{Z}}_{m}^{T}\tilde{\bm{Z}}_{m}\right)^{-1}\tilde{\bm{Z}}_{m}^{T}$
exists, and we denote it as $\bm{P}$. For a matrix $\bm{A}$, define
$\Vert \bm{A} \Vert_2 = \sup_{\bm{x:\Vert x\Vert }= 1} \Vert \bm{A}
\bm{x}\Vert $. Since
$ \Vert \bm{P}_{m}\Vert _{2} \leq \Vert \bm{P}\Vert _{2} + \Vert
\bm{P}_{m}-\bm{P}\Vert _{2}$ and
$\Vert \bm{P}\Vert _{2} \leq \Vert \bm{P}_m\Vert _{2} + \Vert
\bm{P}_{m}-\bm{P}\Vert _{2}$, we have
\begin{equation}
  \label{eq:mtineq}
   \left|\left\Vert \bm{P}_{m}\right\Vert _{2}-\left\Vert
      \bm{P}\right\Vert _{2}\right|\leq\left\Vert
    \bm{P}_{m}-\bm{P}\right\Vert _{2}.
\end{equation}
Since for all $m$, $\left\Vert \bm{P}_{m}\right\Vert _{2} =1,$ being
its largest eigenvalue and
$\left\Vert \bm{P}_{m}-\bm{P}\right\Vert _{2} \rightarrow 0$ as
$m \rightarrow \infty$, \eqref{eq:mtineq} implies that
$\left\Vert \bm{P}\right\Vert _{2} =1$. Thus the maximum eigenvalue of
$\bm{P}$ is 1. Then for any $\bm{\lambda}\in S^{*}\backslash S$ with
$\lim_{m\rightarrow\infty}\bm{\lambda}_{m}=\bm{\lambda}$, we have
\begin{equation}
  \label{eq:lams}
  \lim_{m\rightarrow\infty}\bm{1}^{T}\bm{Z}\left(\bm{Z}^{T}\bm{\Lambda}_{m}^{-2}\bm{Z}\right)^{-1}\bm{Z}^{T}\bm{1}=\lim_{m\rightarrow\infty}\bm{\lambda}_{m}^{T}\tilde{\bm{Z}}_{m}\left(\tilde{\bm{Z}}_{m}^{T}\tilde{\bm{Z}}_{m}\right)^{-1}\tilde{\bm{Z}}_{m}^{T}\bm{\lambda}_{m}=\bm{\lambda}^{T}\bm{P}\bm{\lambda}^{T}.
\end{equation}
Since $\bm{Z}^{T}\bm{e}=\bm{0}$, then
$ \tilde{\bm{Z}}^T_m \bm{\Lambda}_m \bm{e} =
\bm{Z}^{T}\bm{\Lambda}_{m}^{-1}\bm{\Lambda}_{m}\bm{e}=\bm{0}$. Define
$\bm{\Lambda}_{m}\bm{e}=\tilde{\bm{e}}_{m}=\left(\tilde{e}_{m1},\tilde{e}_{m2},\dots,\tilde{e}_{mn}\right)^{T}$,
where $\tilde{e}_{mi}=\lambda_{i,m}e_{i}$ and
$\lim_{m\rightarrow\infty}\tilde{\bm{e}}_{m}=\tilde{\bm{e}}=\left(\lambda_{1}e_{1},\dots,\lambda_{n}e_{n}\right)^T$. So
we have
\[
\bm{P}  \tilde{\bm{e}} = \lim_{m\rightarrow\infty}\tilde{\bm{Z}}_{m}\left(\tilde{\bm{Z}}_{m}^{T}\tilde{\bm{Z}}_{m}\right)^{-1}\tilde{\bm{Z}}_{m}^{T}\bm{\Lambda}_{m}\bm{e} = \bm{0}.
\]
Thus $\tilde{\bm{e}}$ is an eigenvector of $\bm{P}$ corresponding to
eigenvalue 0. We also know that
$\bm{\lambda}^{T}\tilde{\bm{e}}=\sum_{i=1}^{n}\lambda_{i}^{2}e_{i}>0$. So
using similar arguments as before, $\bm{\lambda}$ cannot be an
eigenvector for $\bm{P}$ corresponding to eigenvalue 1. Thus
$\bm{\lambda}^{T}\bm{P}\bm{\lambda}<1$, which by \eqref{eq:lams}
implies $f(\bm{\lambda}) <1$ for any $\bm{\lambda} \in S^* \backslash S$.

Therefore for any $\bm{\lambda}\in S^{*}$, $f\left(\bm{\lambda}\right)<1$. Since $S^{*}$ is a compact set, and $f\left(\bm{\lambda}\right)$ is a continuous function of $\bm{\lambda}$ over $S^{*}$, we have
\[
\sup_{\bm{\lambda}\in S^{*}}f\left(\bm{\lambda}\right)=f\left(\tilde{\bm{\lambda}}\right),\quad\text{for some }\bm{\tilde{\lambda}}\in S^{*}.
\]
Therefore $\sup_{\bm{\lambda}\in S^{*}}f\left(\bm{\lambda}\right)<1$,
which by \eqref{eq:wlam} and \eqref{eq:supremum} in turn implies that
\[
\sup_{\bm{\omega}\in\mathbb{R}_{+}^{n}}\frac{\bm{1}^{T}\bm{Z}\left(\bm{Z}^{T}\bm{\Omega}\bm{Z}\right)^{-1}\bm{Z}^{T}\bm{1}}{\sum_{i=1}^{n}1/\omega_{i}}<1.
\]

Let $\rho_{1}=\sup_{\bm{\omega}\in\mathbb{R}_{+}^{n}}\frac{\bm{1}^{T}\bm{Z}\left(\bm{Z}^{T}\bm{\Omega}\bm{Z}\right)^{-1}\bm{Z}^{T}\bm{1}}{\sum_{i=1}^{n}1/\omega_{i}}$, so we have
\[
\bm{1}^{T}\bm{Z}\left(\bm{Z}^{T}\bm{\Omega}\bm{Z}\right)^{-1}\bm{Z}^{T}\bm{1}\text{\ensuremath{\leq}}\rho_{1}\sum_{i=1}^{n}\frac{1}{\omega_{i}}.
\]

\end{proof}

\bibliographystyle{apalike}
\bibliography{ref_logit_improper}
\end{document}